\documentclass{article}
\usepackage{amsmath, amssymb, amsthm, bbm, amsfonts}
\usepackage{enumitem}
\usepackage{tikz}
\usepackage{comment}
\usepackage[utf8]{inputenc}
\usepackage[UKenglish]{isodate}
\usepackage[font=small,labelfont=it]{caption}
\usepackage{subcaption}
\usepackage{multirow}
\usepackage{fancyvrb}
\usepackage[figturesleft]{rotating}
\usepackage{geometry}
\cleanlookdateon
\newcommand{\etam}{\underline{\eta}}
\DeclareMathOperator{\Leb}{Leb}

\DeclareMathOperator{\diam}{diam}
\newcommand{\n}{\mathbbm{n}}
\newcommand{\cL}{\mathcal{L}}
\newcommand{\PP}{\mathbb P}

\newcommand{\bo}{\boldsymbol}
\newcommand{\N}{\mathbb N}
\newcommand{\R}{\mathbb R}
\newcommand{\Z}{\mathbb Z}
\newcommand{\mc}{\mathcal}
\makeatletter
\newcommand{\optionalitemlabel}[2]{%
  \phantomsection
  #1\protected@edef\@currentlabel{#1}\label{#2}%
}
\makeatother

\newtheorem*{theorem*}{Theorem}
\newtheorem{theorem}{Theorem}[section]

\newtheorem{proposition}[theorem]{Proposition}
\newtheorem{corollary}[theorem]{Corollary}
\newcommand{\B}{\mathcal B}
\newtheorem{remark}[theorem]{Remark}
\newtheorem{lemma}[theorem]{Lemma}

\newtheorem{definition}{Definition}[section]

\title{ Filtering and Statistical Properties of Unimodal Maps Perturbed by Heteroscedastic Noises}

\author{Fabrizio~Lillo\thanks{Dipartimento di Matematica, Universit\'a di Bologna and Scuola Normale Superiore, Pisa, Italy. Email address: fabrizio.lillo@unibo.it.} 
\and Stefano~Marmi\thanks{Scuola Normale Superiore, Pisa, Italy. Email address: stefano.marmi@sns.it.} 
\and Matteo Tanzi\thanks{King's College London. Email address: matteo.tanzi@kcl.ac.uk.}
\and Sandro~Vaienti\thanks{Aix Marseille Universit\'e, Universit\'e de Toulon, CNRS, CPT, 13009 Marseille, France. Email address: vaienti@cpt.univ-mrs.fr.}}
\date{\today}

\begin{document}
\maketitle
\begin{abstract}
We propose a theory of unimodal maps perturbed by a heteroscedastic Markov chain  noise
 and experiencing  another  heteroscedastic noise due to uncertain observation. We address and treat the filtering problem  showing that by collecting more and more observations, one would predict the same distribution for the state
of the underlying Markov chain no matter one’s initial guess. Moreover we   give other limit theorems, emphasizing in particular concentration inequalities and extreme value and Poisson distributions. Our results apply  to a   family of maps arising from a model of systemic risk in finance.
\end{abstract}



\section{Introduction}
The measurement of any quantity is plagued by an observational error that prevents knowing with infinite precision the true value of the variable of interest, which therefore remains latent. In some fields, as for example climate or finance, this measurement error can be very large and the observational error distribution, including its bias and variance, are unknown and not constant in time. In time series analysis, the procedure to obtain an estimate of the true latent process and to provide accurate predictions in the presence of observational noise is termed {\it filtering}. The seminal Kalman filter provides an optimal filtering procedure when the dynamics of the latent process is linear and the noise driving the latent dynamics and the observational noise are Gaussian distributed. 
In many systems, however, the dynamic is strongly non-linear and therefore approximated methods, such as the extended Kalman filter or the particle filter approach, provide solutions that can be used in practice.
From the theoretical perspective, however, the properties of filtering of highly non-linear dynamics are still not fully characterized. In particular, we are interested here in unimodal maps perturbed by a heteroscedastic\footnote{We remind that in an heteroscedastic process the variance is not constant and depends on time either explicitly or via the dynamical variables.} which are observed with an heteroscedastic observational noise. Already without including the observational noise, the theory of unimodal maps perturbed by heteroscedastic noise is quite involved.
In our previous papers \cite{lillo21,lillo22} we
considered a random transformation built upon unimodal maps $T$ defined on the unit interval and perturbed with heteroscedastic noise that arises in some financial models of a bank leverage (as we recall in section \ref{rrr}). Our random perturbation reduced to a Markov chain with a kernel not necessarily strictly positive, and to prove the existence and uniqueness of the stationary measure we had to exploit the topological properties of the map $T$. We were thus able to  show the weak convergence of the unique stationary measure to the invariant measure of the map in the noiseless limit. This step was again 
not trivial because of the singularity of  the stochastic kernel. Third, we
introduced the average Lyapunov exponent by integrating the logarithm of the derivative of the
map $T$ with respect to the stationary measure and showed that the average Lyapunov exponent
depends continuously on the Markov chain parameters. 

To achieve the above, we 
explicitly constructed a sequence of random transformations close to $T$, which allowed us to
replace the deterministic orbit of $T$ with a random orbit.  Representing the Markov chain in terms of
random transformation enabled us to state and prove some important limit theorems, such
as the Central Limit Theorem, the large deviation principle, and the Berry-Ess\'een inequality. Fourth, for the class of unimodal maps $T$ of the chaotic type, we performed a multifractal
analysis for the invariant and the stationary measures. Finally, we developed an Extreme Value
Theory (EVT, henceforth) for our Markov chain and we  proved Gumbel’s law for the Markov chain with an extreme index equal to $1$.\\

We now continue that investigation of this model by adding another ingredient often present in applications and that we have previously neglected: observational noise. It is in fact expected that any observations will have a small uncertainty to the actual state of the Markov chain. We assume that also the observational noise is heteroscedastic, since its intensity is expected to depend on the position of the Markov chain. \\ 

The importance of the observational noise to establish statistical properties of dynamical systems was already emphasized in the paper \cite{FV} by one of us. We quote in that paper what  Lalley and Noble wrote in \cite{LN}:\\
{\em In this model our observations take the form $y_i = T^ix+\eta_i,$
where $\eta_i$ are independent, mean zero random vectors. In contrast
with the dynamical noise model (e.g.; the random transformations), the noise does not interact with the dynamics: the deterministic character of the system, and its long range dependence,
are preserved beneath the noise. Due in part to this dependence,
estimation in the observational noise model has not been broadly
addressed by statisticians, though the model captures important
features of many experimental situations.}\\


We also wrote in \cite {FV} that  a system contaminated by the observational noise raises
the natural and practical question whether it would be possible
to recover the original signal, in our case the deterministic orbit $T^ix, i\ge 1.$ This question is amplified in the current paper to recover the evolution of a Markov chain, arising as the perturbation of a deterministic system perturbed with additive heteroscedastic noise. In this regard we consider our result on filtering particularly relevant. In filtering, one usually needs to assume a certain prior for the initial state. Below we show that, in the system at hand, by collecting more and more {\em observations}, one would predict the
same distribution for the state of the underlying Markov chain no matter one’s initial guess.\\

In the last years a few techniques have been proposed
for such a noise reduction \cite{sei}: we remind here the remarkable
Schreiber–Lalley method \cite{sette,otto,nove,dieci}, which provides a very consistent
algorithm to perform the noise reduction when the underlying
deterministic dynamical system has strong hyperbolic properties.\\

{\bf Plan of the paper.} In the rest of this introduction we recall the construction of the Markov chain, deferring to section 2 for the question of {\em filtering};  Section 3 will show that limit theorems continue to hold in the presence of the observational noise, and we will add to them the concentration inequalities described in section 4. We will finally end with a few recurrence results, like extreme value theory and   and Poisson statistics. 
\subsection{Random transformations} 
 We start with a short reminder on random transformations. Let $(\Omega, \mathcal{T}, \mathbb{P})$ be a probability space and consider the countable product space $(\tilde{\Omega}, \tilde{\mathcal{T}}, \tilde{\mathbb{P}}),$  where $\tilde{\Omega}=\Omega^{\mathbb{N}}, \tilde{\mathcal{T}}=\mathcal{T}^{\otimes \mathbb{N}}, \tilde{\mathbb{P}}=\mathbb{P}^{\otimes \mathbb{N}}.$ Put $S$ the one sided shift over $\tilde{\Omega}.$ Then to  each $\eta\in \Omega$ we associate a map $T_{\eta}$ on the measurable space $(I, \mathcal{A})$ such that $(\eta, x)\rightarrow T_{\eta}(x)$ is measurable and we define $T_{\etam}=T_{\eta_0}$ for each $\etam\in \tilde{\Omega},$ with $\etam=(\eta_0, \eta_1, \dots).$ We then define the random orbit $$T^n_{\etam}:=T_{S^{n-1}(\etam)}\circ \cdots \circ T_{\etam}(x)=T_{\eta_{n-1}}\circ\cdots\circ T_{\eta_0}.$$
If we now consider the skew-product transformation $F:\tilde{\Omega}\times I\rightarrow \tilde{\Omega}\times I$ by $F(\etam, x)=(S\etam, T_{\etam}x),$ we say that a probability measure $\mu_s$ on $(I, \mathcal{A})$ is {\em stationary} if $\tilde{\mathbb{P}}\otimes\mu_s$ is $F$-invariant. It is not difficult to show that $\mu_s$ is stationary iff $\mu_s(A)=\int_{\Omega}\mu_s(T^{-1}_{\eta}(A))d\mathbb{P}(\eta),$ for each $A\in \mathcal{A}.$ If we set
\begin{eqnarray}\label{mmmkkk}
X_{n+1}=T_{\eta_n}(X_n),
\end{eqnarray}
with $X_0$ having a certain distribution, 
this defines a homogeneous Markov chain with state space $(I, \mathcal{A})$ and transition operator
\begin{equation}\label{TO}
K(x,A)=\mathbb{P}(\{\eta: \,T_{\eta}(x)\in A\}),
\end{equation}
for any $x\in I, A\in \mathcal{A}.$ We now suppose that $I$ is a subset of $\mathbb{R}$ with $\mathcal{A}$ the Borel $\sigma-$algebra. Moreover $(I, \mathcal{A})$ will be  endowed with the probability Lebesgue measure $m$\footnote{In the following we will use the shorthand $dm=dx$ and  if it is not otherwise specified, all the $L^p$ spaces are taken with respect
to $m$ and the integrals with respect to such a measure are defined on I.} such that each transformation $T_{\eta}$ is non-singular w.r.t. $m.$
We remind that the transfer operator associated to the map $T_{\eta}$ is defined by the duality relationship $\int\mathcal{L}_{\eta}f\,v\,dx=\int f \ v\circ T_{\eta}\,dx$, where $f \in L^{1}$ and $v \in L^{\infty}$. 

 We then define the averaged operator: for $f\in L^1$
\begin{equation}\label{MMO}
    \mathcal{L}f(x) = \int (\mathcal{L}_{\eta}f)(x)\,d\PP(\eta).
\end{equation}
An absolutely continuous (with respect to $m$) probability measure is stationary iff its density is a fixed point of $\cL.$ We also remind that if the measure given by (\ref{TO}), $K(x, \cdot),$ has a density (Markov kernel), $p(x,\cdot),$ then we also have
\begin{equation}\label{MK}
 \mathcal{L}f(x)=\int f(z) p(z,x)dz.
\end{equation}
This will reveal particularly useful in the next section, where the Markov kernel is explicitly computed. Notice that (\ref{MMO}) and  the duality relationship for the single transfer operators allow us to get a similar relationship for the operator $\cL,$ namely
$$
\int\mathcal{L}f\,v\,dx=\int f\ 
\mathcal{U}v \ dx$$ where 
\begin{equation*}
    (\mathcal{U}v)(x) = \int v(T_{\eta}(x))d\PP(\eta),
\end{equation*}
with $f \in L^1, v \in L^{\infty}.
$
We will assume that $\cL$ has good spectral properties on some Banach space of functions $(\B, \|\cdot\|);$ in particular we will choose  $\B$ as  the space of functions $f$ with total variation $\text{var}(f)<\infty,$ by first setting $\text{Var}(f):=\inf\{\text{var}(v); f=v, m\ a.e.\}$ and then posing
$$
\|f\|_{\B}=\text{Var}(f)+\|f\|_{L^1}.
$$
\subsection{The first heteroscedastic noise}\label{tre}
 The system which we quoted in the Introduction was given by random transformations of the form $T^n_{\etam}:=T_{\eta_{n-1}}\circ\cdots\circ T_{\eta_0},$ where 
 \begin{equation}\label{rrr}
T_{\eta}(x)=T(x)+\sigma(x)\eta
\end{equation}
 and $\eta$ is a  random variable  defined on some probability space $(\Omega, \mathbb{P})$, with values on $\mathbb{R}$ nd with probability distribution $d\Theta(x)=g(x)dx.$ 
We now  better detail the choices of $T, \sigma$ and $\Theta.$ First of all, the map $T$ was chosen as a unimodal map\footnote{See \cite{thunberg2001periodicity} H. Thunberg, for a good review on unimodal maps.} verifying the following conditions
\begin{enumerate}
    \item $T$ is  continuous on the unit interval $I\overset{\text{def}}{=}[0,1]$ with a unique maximum at the point $\mathfrak{c}$ such that $\Upsilon \overset{\text{def}}{=} T(\mathfrak{c}) < 1$.
    \item  There exists a closed interval $[d_1, d_2] \subset I$ which is forward invariant for the map and upon which $T$ and all its power $T^{n},\,n\in\mathbb{N}_{\geq 1}$ are topologically transitive
    \item   $T$ preserves a unique Borel probability measure $\mu$, which is absolute continuous with respect to the Lebesgue measure. 
  \end{enumerate}  
  {\bf Assumptions TM}\\
  In the rest of the paper we will keep the  additional assumptions described below. They are   very precise and one could wonder about the generality of
our achievements. First of all such requirements allow us to get rigorous results; secondly, we will present
in the Appendix a model arising in finance which fits with assumptions TM. We will return on that in the remark below.\\

 We first  introduce $
    \Gamma \overset{\text{def}}{=} 1 - \Upsilon,$
\textit{i.e.}, the gap between $T(\mathfrak{c})$ and $1$. 
Second, we define a positive constant $a$ satisfying 
\begin{equation}\label{eq:boundona2}
    a \leq \frac{1}{\sigma_{\text{max}}}\min\left\{\frac{\Gamma}{2}, \frac{q}{2}, \frac{1}{2}T\left(1-\frac{\Gamma}{2}\right)\right\},
\end{equation}
where $\sigma_{\max} \overset{\text{def}}{=} \max_{x \in [-\Gamma, 1]}\sigma(x)$, and the positive constant $q$ is the  intercept of the map $T$ at zero, $T(0)=q.$ 
The function $\sigma$ is a non-negative and  differentiable function for $x\in \mathbb{R}.$ 
We then take the density $g$ as
\begin{equation}\label{eq:gaussiankernel}
   g(x)\overset{\text{def}}{=} c_{a} \chi_{a}(x)e^{-\frac{x^2}{2}},\quad x \in \mathbb{R},
\end{equation}
where $\chi$ is a $C^{\infty}$ bump function and
\begin{equation*}
    c_a = \left(\int \chi_{a}(x)e^{-\frac{x^2}{2}}\,dx\right)^{-1}.
\end{equation*}
 Since the noise varies in a neighborhood of $0$, we need to enlarge the domain of definition of the map $T$ to take into account the action of the noise. More precisely, we extend the domain of $T$ to the larger interval $\tilde{I} \overset{\text{def}}{=} [-\Gamma, 1]$. On the interval $[-\Gamma,0]$, $T$ has a  $C^4$ extension and is 
  decreasing with $T(-\Gamma) < \Delta.$ 
 With abuse of language and notation, we will continue to call $T$ the map after its redefinition, and we put $I = \tilde{I}$. We  proved in \cite{lillo22} that the random  orbit $T^n_{\etam}$ is confined in the interval $[-\Gamma/2, 1-\Gamma/2]$ for any $n$ and  $\etam.$ Moreover there exists only one absolutely continuous stationary measure $\mu_s $
with the support verifying $\text{supp}(\mu_s) \subset I_{\Gamma}$, where 
\begin{equation}\label{eq::supportone}
    I_{\Gamma}\overset{def}{=}\left[\frac{1}{2}T\left(1-\frac{\Gamma}{2}\right), 1-\frac{\Gamma}{2}\right].
    \end{equation}
  The particular form of the noise allows us to write the transfer operator as in (\ref{MK}), $
 \mathcal{L}f(z)=\int f(x) p(x,z)dx,$
with the explicit  Markov kernel:
\begin{equation}\label{MMKK}
p(x,z):=\frac{1}{\sigma(x)}g\left(\frac{z-T(x)}{\sigma(x)}\right)=\frac{1}{\sigma(x)}c_a\chi_a\left(\frac{z-T(x)}{\sigma(x)}\right)e^{-\frac{(z-T(x))^2}{2\sigma^2(x)}}
\end{equation}
    The above considerations imply that the subspace $\mathcal{S}_{I_{\Gamma}}:=\{h \in L^{1}\,:\,\text{supp}(h) \subset I_{\Gamma}\}$  is $\mathcal{L}$-invariant, and that the stochastic kernel $p(x,z)$  has total variation of order $\frac{1}{\sigma(x)}$ for any $z\in I_{\Gamma};$ then $\text{Var}(|p(z,\cdot)|)\in L^{\infty}.$
In this case and writing $C_v$ for the essential upper bound of $\text{Var}(|p(z,\cdot)|)$, it is easy to show that, see Lemma B1 in \cite{lillo21},
$$
\|\cL(f)\|_{\B}\le (C_v+1)\|f\|_{L^1}.
$$
This implies that the operator $\cL$ is compact\footnote{The argument is the following. Since  $\B$ is compactly embedded in $L^1,$  for any   sequence $f_n$ in $\B$ such that $\|f_n\|_{\B}\le 1$, there exists a subsequence $f_{n_k}$ and $g\in \B$ such that $\|f_{n_k}-g\|_{L^1}\rightarrow 0.$
To prove  compactness of $\cL$ it will be enough to show that 
 given the sequence $f_n$  with $\|f_n\|_{\B}\le 1$, then the sequence  $\{\cL(f_n)\}_n$ contains a subsequence converging in $\B.$
But $\|\cL(f_n)\|_{\B}\le (C_v+1)\|f_n\|_{L^1}.$
Passing to the subsequence identified above,  we get 
$\|\cL(f_{n_k})-\cL(g)\|_{\B}=\|\cL(f_{n_k}-g)\|_{\B}\le (C_v+1) \|\cL(f_{n_k}-g)\|_{L^1}\le (C_v+1) \|f_{n_k}-g\|_{L^1}\rightarrow 0.$},
with $1$  as a simple eigenvalue by the uniqueness of the stationary measure.        \begin{remark}
 All the limit theorems  proved in the next sections deeply rely on the compactness of the Markov operator $\mathcal{L};$  actually the {\em quasi-compactness}, see \cite{HH} for the precise definition, of such operator would be enough. In both cases the operator will have a spectral gap insuring exponential decay of correlations for observables in $\mathcal{B}.$ To get such a property and the simplicity of the largest eigenvalue which implies the uniquness of the stationary measure,  we should combine the choice of unimodal maps satisfying the condition 3 above (which qualifies them as {\em chaotic}, see \cite{thunberg2001periodicity}), with  rather general additive noises modulated by the position of the point. That makes delicate  the spectral analysis of $\mathcal{L}$ and one needs sufficient assumptions to achieve it. We believe that large  generalizations are possible following the scheme proposed in this paper.
\end{remark}

\section{The observational noise and the second heteroscedastic perturbation: filtering}\label{filll}
Suppose at time $n$ the system is the configuration $T^n_{\etam}:=T_{\eta_{n-1}}\circ\cdots\circ T_{\eta_0};$ at this point we add another heteroscedastic noise constructed by multiplying a function $s(T_{\eta_n} \circ\ldots\circ T_{\eta_1}(x))$ with an i.i.d. random process  $\{\epsilon_n\}_{n\ge 1}$ defined on the probability space $(\Omega, \mathbb{P})$ with values in $[-1,1]$ independent of the $\eta_k$ and with distribution $\psi.$ We will suppose that   $s:I\rightarrow I$ is a strictly positive Lipschitz function. Put:
\begin{equation}\label{rfu}
Z_n=T^n_{\etam}+s(T^n_{\etam})\epsilon_n.
\end{equation}
We will also suppose that the magnitude of the observational noise  is small enough to leave the value of $Z_n, n>0$ in the interval $I.$ By referring to section \ref{tre} it will  be enough to take
\begin{equation}\label{56}
\sup_{x\in I, \epsilon\in [-1,1]}|s(x)\epsilon|<\Gamma/2.
\end{equation}
 We call $s$ the {\em modulation} term: other assumptions on this term will be added in the next sections.
The variable $Z_n$ defines a stationary process on the probability space $\left((I\times \tilde{\Omega}\times \tilde{\Omega}), (\mu\times \tilde{\mathbb{P}}\times \tilde{\mathbb{P}})\right),$ with values in $I$ and distribution
$
\rho_{Z}
$ satisfying, for any Borel set $A\subset I$ and denoting $\text{id}$ the identity function,
\begin{eqnarray}
\rho_{Z}(A)=\iiint 1_A(T^n_{\etam}(x)+sT^n_{\etam}((x))\epsilon_n)h(x)d\tilde{\mathbb{P}}(\etam)dx d\tilde{\mathbb{P}}(\underline{\epsilon})=\\
\iint \mathcal{U}^n(1_A\circ (\text{id}+s(\text{id})\epsilon)) (x)h(x)dxd\psi(\epsilon)=
\iint 1_A(x+s(x)\epsilon)\ \mathcal{L}^n(h)(x)dxd\psi(\epsilon)=\\
\int d\mu\int 1_A(x+s(x)\epsilon)d\psi(\epsilon)=\int \psi\left(\frac{A-x}{s(x)}\right)d\mu(x),
\end{eqnarray}
where $h\in \mathcal{S}_{I_{\Gamma}}$ is the density of $\mu.$ In the following we will set
$$
\mathfrak{S}:=I\times \tilde{\Omega}\times \tilde{\Omega}; \ \mathfrak{P}:=\mu\times \tilde{\mathbb{P}}\times \tilde{\mathbb{P}},
$$
and $\rho_Z$ will be the common distribution of  $Z_n.$\\
\begin{remark}\label{Rem:ChoiceOfP}
    It is important to point out that the choice of $\mathfrak{P}$ as the underlying probability to study processes like $\{\phi(Z_n)\}_n, \phi \in \mathcal{B}$ and the consequent stationarity of the distribution $\rho_Z,$ allow us to get the limit theorems of Sections \ref{LLTT}, \ref{EEMM}, \ref{EEVVTT}, \ref{PPSS}. See also remark \ref{staz} for the stationarity in the case of filtering.
\end{remark}
We now observe that  $Z_n$ is the sum of the Markov chain $X_n$ defined in (\ref{mmmkkk}), plus the noise $s(X_n)\epsilon_n.$  If we equip the chain with the initial distribution $\hat{w},$ then the distribution $\rho_{Z_n}$ of $Z_n$ will satisfy
$\rho_{Z_n}(A)=\iint 1_A(X_n+s(X_n)\epsilon)d\mathbb{P}_{\hat{w}}d\psi(\epsilon),$ which is independent of $n$ and equal to $\rho_Z$ above when $\hat{w}=\mu.$\\

We say that $\Pi_n$ is a {\em filtering process}  if it is a sequence of  conditional probability distributions  on $I$ such that, $\mathbb{P}$-almost surely, for $n\ge 0$
$$
\Pi_n(B):=\mathbb{P}(X_n\in B|Z_0,Z_1,\dots, Z_n), \forall B\in \mathcal{B}.
$$
This is a conditional probability and therefore a random variable defined as the conditional expectation\footnote{We write $\mathbb{E}_{\mathbb{P}}$ for expectation with respect to the probability measure $\mathbb{P}$. When no ambiguity arises, we may omit the subscript.}
$$
\mathbb{P}(X_n\in B|Z_1,\dots, Z_n)=\mathbb{E}_{\mathbb{P}}(1_B(X_n)|Z_0,Z_1,\dots. Z_n).
$$
On the other hand, the application $B\rightarrow \mathbb{P}(X_n\in B|Z_1,\dots, Z_n)$ defines a Borel measure on $I.$
We will assume that \footnote{As it is written in \cite{CH} ``the
deterministic choice (\ref{ggg}) is only a matter of convenience; it means that all the
information about $X_0$ is contained in its a priori distribution $\hat{w}$".}   
\begin{equation}\label{ggg}
\mathbb{P}(Z_0=0)=1
\end{equation}
and 
$$
\Pi_0(B)=\hat{w}(B)=\mathbb{P}(X_0\in B),
$$
where $\hat{w}$ is the initial distribution of the Markov chain $X_n$ which we will suppose having density $w\in \mathcal{B}.$ 
The problem of calculating $\Pi_n$ is called nonlinear filtering and the  the filtering process
can be calculated in an iterative fashion.  To achieve this, we first assume a homogeneous memoryless channel. This means
\begin{equation}\label{ce}
\mathbb{P}(Z_1\in B_1,\dots,Z_n\in B_n|X_1,\dots,X_n)=\prod_{k=1}^n\mathbb{P}(Z_k\in B_k|X_k).
\end{equation}
If we assume, as we did, that the $\epsilon_j$ are i.i.d. and also independent of the $X_j,$ it is easy to show that (\ref{ce}) holds in our case. Moreover we have explicitly
$$
\mathbb{P}(Z_k\in A|X_k)=\int 1_{A}(X_k+s(X_k)\epsilon)d\psi(\epsilon),
$$
and it is also very easy to check that the distribution of the random variable $\mathbb{P}(Z_k\in A|X_k)$ does not depend upon $k$ (use the stationarity of the Markov chain). \\


If we now assume that $\psi$ has density $\psi'$ and support contained in the interval $[-\epsilon, \epsilon]$, we have, for any $k:$
$$
\mathbb{P}(Z_k\in A|X_k)=\int_A g(y, X_k)dy,
$$
where
$$
g(y,x)=\frac{1}{s(x)}\psi'\left(\frac{y-x}{s(x)}\right) 1_{[x-s(x)\epsilon,x+ s(x)\epsilon]}(y) .
$$
This is called the {\em likelihood function} in \cite{BDM}. Then by Proposition 3.2.5 in \cite{CMR} we have that:
\begin{eqnarray}
\Pi_n(\phi)=\frac{\int_I\phi(x)g(Z_n,x)d\Pi^{+}_{n-1}(x)}{\int_Ig(Z_n,x)d\Pi^{+}_{n-1}(x)}\label{Eq:Filt1}\\
\Pi^{+}_n(\phi)=\int K(z,\phi)\Pi_n(dz),\label{Eq:Filt2}
\end{eqnarray}
where $\phi$ is a measurable bounded function and $K$ is the Markov kernel of $X_k.$ 
\subsection{Filtering via Random Cones}
In this subsection we are going to prove that under suitable conditions, the filtering problem defined by the equations \eqref{Eq:Filt1}-\eqref{Eq:Filt2} admits a unique equivariant measure and for any choice of initial prior, there is convergence to this equivariant measure. This corresponds to well-posedness of the filtering problem; in words, it means that by collecting more and more observations, one would predict the same distribution for the state of the underlying Markov chain no matter one's initial guess.

To do so, we first express the evolution in equations \eqref{Eq:Filt1}-\eqref{Eq:Filt2} in terms of a cocycle. Then we restrict the action of the cocycle to cones that depend on the fiber where they are defined, and that are therefore random. We then give sufficient conditions for the cocycle to act as a contraction with positive frequency.

\subsubsection{Definition of the cocylce}
Recall that the process $(X_n,Z_n)_{n\in \N}$ is defined on the probability space $\left(\mathfrak{S}, \mathfrak{P}\right)$ with values in $I\times I.$ We remind that  $\mu$ is the stationary measure for the chain $(X_n)_{n\in \N}$ and prescribes the distribution of $X_0$. We first notice that by the Kolmogorov Extension Theorem, we get a double sided extension for our process $(X_n,Z_n)_{n\in \Z}$ having the same transition probability and same distribution for $(X_0,Z_0)$. Pushing forward the measure $\mathfrak{P}$ under the mapping $\Omega\mapsto (X_n(\Omega),Z_n(\Omega))_{n\in \Z}$, gives a measure $\bo \mu$  on the space of double sided sequences $(\bo x,\bo z)\in (I\times I)^\Z$.  It readily follows that this measure is invariant and ergodic under the left shift $\sigma$ (compare also with Remark \ref{Rem:ChoiceOfP}).

We want to make explicit the dependence of the evolution $\Pi_{n-1}\mapsto \Pi_{n}$ on $Z_n$. Given a measure $\nu$ on $I$ (e.g.   $\Pi_{n-1}(\Omega)$) and $z\in I$ (e.g.   $Z_{n}(\Omega)=z$), define
\begin{equation}\label{Eq:Evolution}
(\mc P_z\nu)(\phi)= \frac{\int _I\phi(x)g(z,x)d\nu^+(x)}{\int _Ig(z,x)d\nu^+(x)}\quad\quad \nu^+(\phi)=\int K(x,\phi)d\nu(x).
\end{equation}
With this notation, the evolution from $\Pi_{n-1}$ to $\Pi_{n}$ in Eq. (18)-(19) can be written as $\Pi_{n}=\mc P_{Z_n}(\Pi_{n-1})$. Then, we can construct a cocycle $\mc P$ on $(I\times I)^\Z\times \mc M(I)$, with base $(I\times I)^\Z$ and fiber $\mc M(I)$  denoting the space of Borel finite measures on $I$:
\begin{equation}\label{Eq:CocycleDef}
\mc P((\bo x,\bo z), \nu)=(\sigma(\bo x,\bo z),\mc P_{(\bo x,\bo z)}\nu).
\end{equation}
where from now on we will denote $(\bo x, \bo z)_0=z_0$, $\mc P_{(\bo x,\bo z)}=\mc P_{z_0}$ and \[
\mc P^n_{(\bo x,\bo z)}=\prod_{i=0}^{n-1}\mc P_{\sigma^i(\bo x,\bo z)}=\prod_{i=0}^{n-1}\mc P_{z_i}.
\]

The strategy is to restrict the action on the fiber to equivariant $(\bo x,\bo z)$-dependent cones of functions, that will be specified below,  so that\footnote{Here will play an important role the properties of the transition probabilities of $X_n$ and the noise $\epsilon_n$. } \[
\mc P_{(\bo x,\bo z)}:(\mc V_{(\bo x,\bo z)},\Theta_{(\bo x,\bo z)})\rightarrow (\mc V_{\sigma (\bo x,\bo z)},\Theta_{\sigma (\bo x,\bo z)})\]
are  contractions with positive frequency for $\bo \mu$-a.e. $(\bo x,\bo z)$ with respect to $\Theta_{(\bo x,\bo z)}$, the Hilbert projective metric intrinsically defined on $\mc V_{(\bo x,\bo z)}$.
More precisely, we are aiming for the following 
\begin{definition}[Attracting Equivariant Measure]For almost every sequence $(\bo x,\bo z)$ there is $\hat \mu_{(\bo x,\bo z)}$ in $\mc V_{(\bo x,\bo z)}=\mc V_{z_0}$ such that
\begin{enumerate}
\item $\mc P_{z_0}\hat \mu_{(\bo x,\bo z)}=\hat \mu_{\sigma(\bo x,\bo z)}$;
\item there are $C>0$ and $\lambda\in(0,1)$ such that for all $n\in \N$
\[
 \Theta_{(\bo x,\bo z)}\left(\mc P^n_{\sigma^{-n}(\bo x,\bo z)}\nu,\hat \mu_{(\bo x,\bo z)} \right)\le C\lambda^n
\]
 for every $\nu\in \mc V_{\sigma^{-n}(\bo x,\bo z)}$.
\end{enumerate}
\end{definition}

\begin{remark}\label{staz}
 Notice that in our case the equivariant measure $\hat\mu_{(\bo x,\bo z)}$ will only depend on the $\bo z$ variable (and more precisely on its coordinates $\le 0$). We keep an artificial  dependence on $\bo x$ and construct the skew-product on fibres over $(\bo x,\bo z)$ because we have good statistical properties for the measure $\bo \mu$ for the joint variables $(\bo x, \bo z)$ -- namely we can exploit that $(X_n,Z_n)$ is a stationary ergodic process -- but not for $\bo z$. 
\end{remark}

\subsubsection{Cone Construction and Equivariant Measure}

Now we show that under some conditions on the transition kernel of the Markov chain and on the observational noise, the system admits an equivariant measure. Before that, let's gather some useful facts about the cones we are going to use.

\paragraph{Cones of positive functions}

For $c\in(0,1)$, we define 
\[
\mc V_c(J):=\left\{\phi: J\rightarrow \R^+:\,\phi\in L^1,\,\, c < \frac{\phi(x)}{\phi(y)} < c^{-1}  \quad\forall x,y\in J\right\}.
\]

Let $\mc V_0(J)$ be the cone of integrable nonnegative functions defined on $J$. 
Let's call $\Theta_{0}$ the intrinsic Hilbert metric on $\mc V_0(J)$. Notice that this cone has the characteristic that, not only its diameter with respect to the Hilbert metric is infinite, but some densities are at infinite distance from each other. Let's recall the definition of the Hilbert metric on a cone, and write an explicit expression for  $\Theta_{0}$ on $\mc V_{0}$. First of all, for any $\phi_1$, $\phi_2\in \mc V_{0}$, define
\[
\beta_{0}(\phi_1,\phi_2)=\inf\{t>0:\,t\phi_1- \phi_2\in \mc V_{0}\}.
\]
Then 
\[
\Theta_0(\phi_1,\phi_2)=\log[\beta_0(\phi_1,\phi_2)\beta_0(\phi_2,\phi_1)].
\]
Notice that if there is $x\in J$ such that $\phi_1(x)=0$, but $\phi_2(x)> 0$, then the set $\{t>0:\,t\phi_1- \phi_2\in \mc V_{0}\}$ is empty, and we use the convention that the infimum of the empty set is $+\infty$. In contrast, suppose that $\{t>0:\,t\phi_1- \phi_2\in \mc V_{0}\}$ is nonempty, then since there must be $x$ for which $\phi_2(x)\neq 0$ and $\phi_1(x)\neq 0$, $t\ge \phi_2(x)/\phi_1(x)>0$ and therefore  $\beta_{0}(\phi_1,\phi_2)>0$. We can summarize this remark by saying that for every $\phi_1,\,\phi_2\in \mc V_{0}$
\[
\beta_{0}(\phi_1,\phi_2)\in(0,+\infty]
\]
which implies  $\Theta_{0}: \mc V_{0}\times  \mc V_{0}\rightarrow[0,+\infty] $
is a Hilbert metric that can take values on the extended positive semiaxis.

The following result is standard in the case of cones with finite Hilbert metrics, we double check that it still holds for the extended metric on the  cone $\mc V_{0}$.
\begin{lemma}\label{Lem:ContHilb}
Let $\mc V$ and $\mc V'$ be  two cones such that their associated Hilbert metrics are $\Theta_*(\psi_1,\psi_2)=\log[\beta_*(\psi_1,\psi_2)\beta_*(\psi_2,\psi_1)]$ for $*\in\{\mc V,\mc V'\}$ and where
\[
\beta_*(\psi_1,\psi_2):=\inf\{t>0:\,t\phi_1- \phi_2\in *\}\in(0,+\infty] \quad\quad\forall \psi_1,\psi_2\in *.
\]
Then if $\mc L:\mc V\rightarrow \mc V'$ is a linear map, one has
\[
\Theta_{\mc V'}(\mc L\phi_1,\mc L \phi_2)\le \Theta_{\mc V}(\phi_1,\phi_2)
\]
\end{lemma}
\begin{proof}
\[
\{t>0:\,t\mc L\phi_1- \mc L\phi_2\in \mc V'\}\supset \{t>0:\,t\phi_1- \phi_2\in \mc V\}
\]
and therefore
\[
\beta_{\mc V'}(\mc L\phi_1,\mc L\phi_2)\le \beta_{\mc V}(\phi_1,\phi_2)
\]
and the statement easily follows.
\end{proof}

The following classic theorem by G. Birkhoff provides conditions for a linear mapping between cones to be a strict contraction.
\begin{theorem}\label{Thm:ContCones}
Assume $\mc V\subset V$ and $\mc V'\subset V'$ are two  cones with Hilbert metrics $\Theta_{\mc V}$ and $\Theta_{\mc V'}$ respectively. If $\mc P: V\rightarrow  V'$ is a linear transformation such that $\mc P(\mc V)\subset \mc V'$, then
\[
\Theta_{\mc V'}(\mc P\psi_1,\mc P\psi_2)\le [1-e^{-\diam_{\mc V'}(\mc P\mc V)}] \Theta_{\mc V}(\psi_1,\psi_2)
\]
for all $\psi_1,\psi_2\in\mc V$,  where $\diam_{\mc V'}(\mc P\mc V)$ indicates the diameter of $\mc P\mc V$ in $\mc V'$.
\end{theorem}

Now, in the case of cones of positive functions, we can check that for $c>0$, $\mc V_c\subset \mc V_0$ has finite diameter in $\mc V_0$ and so if a linear application maps $\mc V_0$ into $\mc V_c$ then it's a contraction with respect to the Hilbert metric $\Theta_0$.

\begin{lemma}\label{Lem:FiniteDiam}
For $c\in(0,1]$, $\mc V_c(J)$ has finite diameter in $(\mc V_0(J),\Theta_0)$.
\end{lemma}
\begin{proof}
For $\phi,\,\psi\in \mc V_c(J)$, it is not hard to compute 
\[
\beta_0(\phi,\psi)\beta_0(\psi,\phi)=\sup\left\{\frac{\phi(x)}{\psi(x)}\frac{\psi(y)}{\phi(y)}:\,x,y\in J\right\}\le c^{-2}.
\]
Therefore, $\Theta_0(\phi,\psi)<\log c^{-2}$.
\end{proof}
The following corollary is now immediate.
\begin{corollary}\label{Cor:Contraction}
If $\mc P$ is a linear transformation with $\mc P\mc V_0(J)\subset \mc V_c$ with $c\in(0,1]$, then
\[
\Theta_0(\mc P\phi,\mc P\psi)\le \lambda\Theta_0(\phi,\psi)\quad\quad\forall\phi,\psi\in\mc V_0
\]
with $\lambda=1-e^{-\diam_{\mc V_0(\mc V_c)}}\in[0,1)$.
\end{corollary}

 \subsubsection{Main Result}
Calling $P(x,A)$ the transition probability of the Markov process $X_n$, we assume that it is absolutely continuous and we denote as in the previous sections
\[
p(x,y)=\frac{dP(x,\cdot)}{d\Leb}(y)
\]
the density for the transition probability. We also call $J_z\subset I$ the support of $g(z,\cdot)$ and $\mu_Z$ the projection of the stationary measure for $(X_n,Z_n)$ on the $Z$ component. We introduce now an abstract assumption that we'll later show  to be satisfied by our model when the dynamical and observational noise satisfy  certain relationships.\\

{\bf Main Assumption:}
There is $c>0$ and  $I'\subset I$ with positive $\mu_Z$ measure\footnote{The stationarity of the process implies that we expect to observe $Z_n$ in the interval $I'$ with positive probability} such that for almost every $z\in I'$ there is a nontrivial interval $F_z\subset I$ satisfying: 
\begin{itemize}
\item[1)] for a.e. $x\in J_z$, $p(x,y)>c$ for every $y\in F_z$;
\item[2)]  there is a non-trivial interval $F_z'\subset F_z $ such that for every $z'\in F_z'$, $J_{z'}\subset F_z$ and $g(z',y)>c$ for every $y\in J_{z'}$;
\item[3)] $p<c^{-1}$ and $g<c^{-1}$.
\end{itemize}
\
\begin{figure}[h!]
\begin{center}
\includegraphics[scale=0.35]{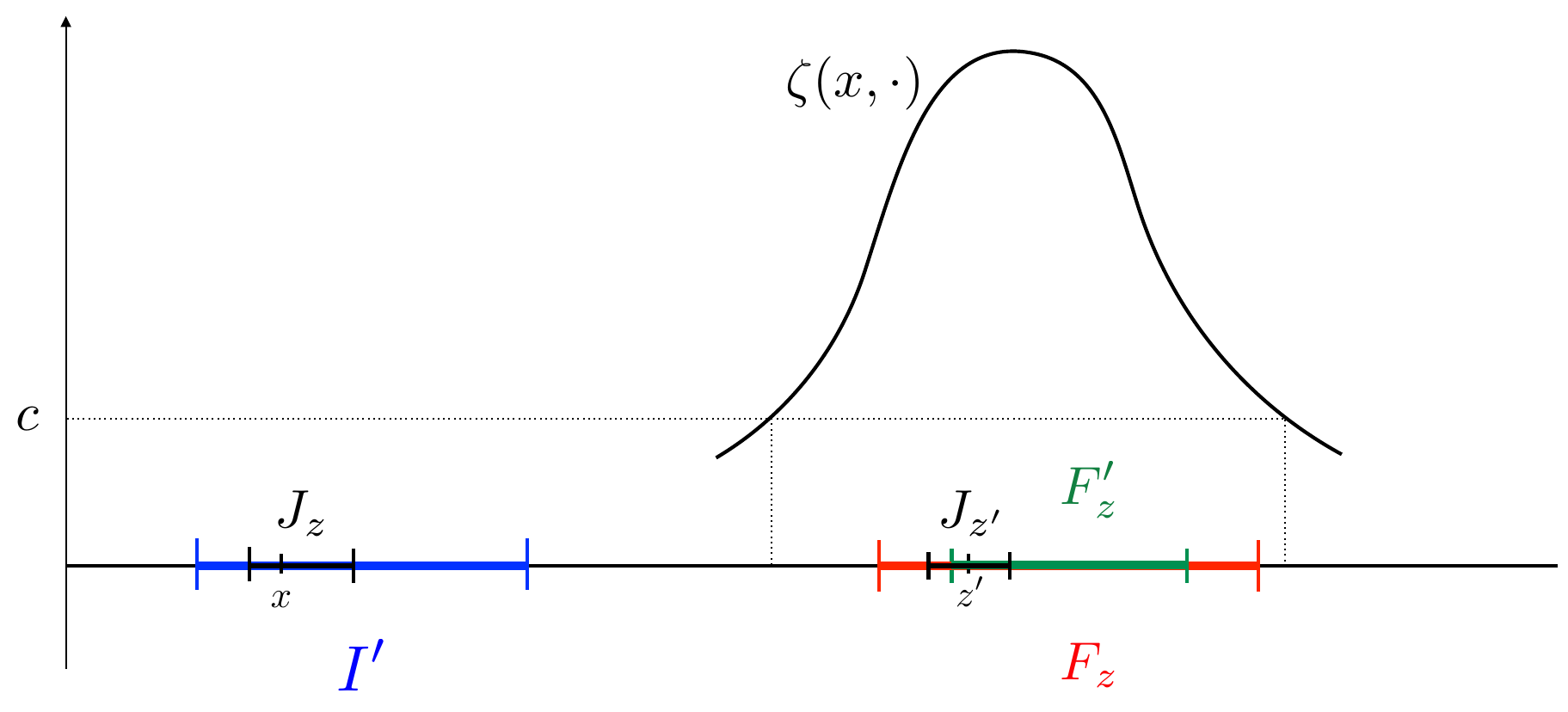}
\caption{The image above depicts the typical situation one has under the Main Assumption. The horizontal axis represents a subset of $I$, while the vertical axis gives the values taken by the transition densities $p(x,\cdot)$. Almost every point $x\in J_z$, where $z\in I'$, has transition density larger than $c>0$ on $F_z$, and every point of $F_z'\subset F_z$ has the corresponding set $J_{z'}$ contained in $F_z$.}
\end{center}
\end{figure}

The above entails  that when an observation ends in $I'$, which by the assumption happens with positive probability and therefore with positive frequency, it sees dynamical noise that is (in some sense) larger than the observational noise. This implies, roughly speaking, that with positive asymptotic frequency, $z_n$ belongs to  a set where $\mc P_{z_n}\mc V_0(J_{z_n})\subset \mc V_{a}(J_{z_{n+1}})$ with $a>0$, and  $\mc P_{z_n}|_{\mc  V_0(J_{z_n})}$ is a uniform contraction from $(\mc V_0(J_{z_n}),\Theta_0)$ to $(\mc V_0(J_{z_{n+1}}),\Theta_0)$ which will lead to the existence of an equivariant measure. We are going to make this more precise below.

Let's consider cones on each fiber, $\mc V_{(\bo x,\bo z)}:=\mc V_0(J_{z_0})$ where recall that $J_{z_0}$ is the support of $g(z_0,\cdot)$.
\begin{proposition}
If the Main Assumption holds, there is a unique absolutely continuous equivariant probability measure for the cocycle in \eqref{Eq:CocycleDef}.
\end{proposition}
\begin{proof}
 With the above definitions, $\mc P_{(\bo x, \bo z)}$ sends nonnegative functions to nonnegative functions, i.e. it maps $\mc V_0(J_{z_0}) $ to $\mc V_0(J_{z_1})$.

Let $z_0\in I'$ and $z_1\in F_{z_0}'$ so that $J_{z_{1}}\subset F_{z_0}$, where $I'$, $F_{z_0}'$, and $F_{z_0}$ are as in the Main Assumption. Item 1) of the main assumption implies that given $\nu$ with density $\psi\in \mc V_0(J_{z_0})$, $\nu^+$ defined as in \eqref{Eq:Evolution} has density $\psi^+(y)=\int p(x,y)\psi(x)dx$ and  on $F_{z_{0}}$, $\psi^+>c$.  Since by assumption 2) we have that $g(z_{1},\cdot)>c$, for every $x,y\in J_{z_1}$, and by assumption 3) $g,\,\psi^+<c^{-1}$
\[
c^4<\frac{g(z_{1},x)\psi^+(x)}{g(z_{1},y)\psi^+(y)} < c^{-4}
\]
which means
\[
g(z_{1},\cdot)\psi^+(\cdot)\in \mc V_{c^4}(J_{z_1}).
\]
Therefore the linear application $\psi\mapsto g(z_{1},\cdot)\psi^+(\cdot)$ sends $\mc V_0(J_{z_0})$ to $\mc V_{c^4}(J_{z_{1}})$, by Corollary \ref{Cor:Contraction} it's a contraction with respect to the Hilbert metric $\Theta_0$. Since when $\nu$ has density $\psi$, by definition, $\mc P_{(\bo x,\bo z)}\nu$ has density
\[
\frac{g(z_{1},\cdot)\psi^+(\cdot)}{\int_Ig(z_{1},y)\psi^+(y)dy}
\]
and the Hilbert metric is projective, then also $\mc P_{(\bo x,\bo z )}$ from $\mc V_0(J_{z_0})$ to  $\mc V_{c^4}(J_{z_1})\subset\mc V_{0}(J_{z_1})$ is a contraction with rate $\lambda\in(0,1)$ independent of $z_0\in I'$.

Now consider the event $E:=\{Z_n\in I' \mbox{ and } J_{Z_{n+1}}\subset F_{Z_n} \}$. The main assumption implies that $E$ has positive measure and, by ergodicity, this event will occur almost surely for infinitely many $n\in \Z$ and with positive frequency.  This implies that for $\bo \mu$-a.e. sequence $(\bo x, \bo z)$ the set
\[
\bigcap_{n\ge 1} \mc P^n_{(\sigma^{-n}\bo x,\sigma^{-n}\bo z)} \mc V_0(J_{z_{-n}})\
\]
contains a unique direction: the intersection is non-empty because intersection of compact (with respect to the weak topology) projective sets; the direction is unique, because the diameter of $\mc P^n_{(\sigma^{-n}\bo x,\sigma^{-n}\bo z)} \mc V_0(J_{z_{-n}})$ tends to zero for a.e. $(\bo x, \bo z)$. The probability density in this unique direction is the equivariant measure in the fibre sitting on $(\bo x, \bo z)$ (the equivariance is immediate by construction).
\end{proof}

Recall that in the model for a bank leverage with heteroskedastic noise, we have that $(X_n,Z_n)$ evolve as
\[
X_{n+1}=T(X_n)+\sigma(X_n) \eta_n \quad Z_{n+1}=X_{n+1}+s(X_{n+1})\epsilon_{n+1}
\]
where $\{\eta_n\}_{n\in \N}$ and $\{\epsilon_n\}_{n\in \N}$ are independent random variables. The random variables $\{\eta_n\}_{n\in \N}$ and the function $\sigma: I\rightarrow \R^+$ are fixed by the model and satisfy: 
\\

(M1) $\{\eta_n\}_{n\in \N}$ are identically distributed with absolutely continuous distribution supported on a centered interval $[-\delta/2,\delta/2]$ and density bounded and lower bounded away from zero; 

(M2) the function $\sigma$ is continuous and lower bounded away from zero. 
\\

For the observational noise, we impose that:\\

{\bf Assumption on the Observational Noise:} $\{\epsilon_n\}_{n\in \N}$ are independent and identically distributed, absolutely continuous, with distribution having support on the interval $[-\epsilon/2,\epsilon/2]$, and positive density uniformly bounded and uniformly bounded away from zero; the function $s:I\rightarrow \R^+$ is continuous and bounded away from zero. 
\\

Notice that in the model for a bank leverage with heteroskedastic noise, under the assumption above, we can write
\begin{align*}
J_z&=\left\{x\in I:\,\exists w\in[-\epsilon/2,\epsilon/2]\,\,\mbox{s.t.}\,\, z=x+s(x)w\right\}\\
&=\left\{x\in I:\,z\in[x-s(x)\epsilon/2,x+s(x)\epsilon/2]\right\}
\end{align*}
in particular, the set $J_z$ is closed.

\begin{lemma}
Let the Assumption on the Observational Noise hold. Assume also that there is $z_0$ in the support of $\mu_Z$  such that 
\begin{itemize}
\item[A)]  for every $x\in J_{z_0}$, $|T(J_{z_0})|<\sigma(x) \frac{\delta}{2}$\footnote{Here $|\cdot|$ applied to a set, denotes the minimal length of an interval containing that set.};
\item[B)]  there is $F'$ an interval around $T(z_0)$ such that for every $z'\in F'$, $J_{z'}$ is contained in $\mbox{Op}(T(J_{z_0}))\neq \emptyset$.
\end{itemize}

Then the model for a bank leverage with heteroskedastic noise satisfies the Main Assumption.
\end{lemma}
\begin{proof}
Condition A) together with the modelling assumption (M1) above ensures that there is $c>0$ such that the kernel density $p(x,y)>c$ for any $x\in J_{z_0}$ and any $y\in T(J_{z_0})$. By continuity of $\sigma$ and $T$, the above is true for any $z$ in a sufficiently small interval $I'$ around $z_0$, i.e. for every $z\in I'$, $x\in J_z$ and $y\in T(J_{z_0})$, $p(x,y)>c>0$ (where $c>0$ is eventually different to the constant found above). Since $z_0$ is in the support of $\mu_Z$ and $I'$ is a nieghborhood of $z_0$, $\mu_Z(I')>0$.  This implies condition 1) in the Main Assumption with $F_z$ independent of $z$ and equal to $F:=T(J_{z_0})$.

Condition B) immediately implies condition 2) of the Main Assumption. Condition 3) of the Main Assumption holds from the modelling hypotheses an the Assumption on the Observational Noise.
\end{proof}

\begin{remark}
A) and B) above are satisfied if the global condition
\[
\epsilon \max_{x\in I} s(x)\le \min\{|I|/10,\min_x\sigma(x)\delta/10\}
\] 
is satisfied.
\end{remark}


\section{Limit theorems}\label{LLTT}
As we did in \cite{lillo22}, we could now use the Nagaeev-Guiva'rch technique to get a few limit theorems, in particular the central limit theorem and the large deviations property.
 Let us first introduce for the observable $u\in {\mathcal B}$:
$$
S_n:=\sum_{k=0}^{n-1}u(T^k_{\overline{\eta}}x+s(T^k_{\overline{\eta}}x)\epsilon_k).
$$
Then we introduce the perturbed operator, for $z\in \mathbb{C}:$
 $$
\mathcal{L}_z(f):=\mathcal{L}(\int e^{zu(\cdot+s(\cdot)\epsilon) }d\psi(\epsilon)f(\cdot)), \  f\in \mathcal{B}.
$$
Notice that $\mathcal{L}_z$ maps again  the Banach space $\mathcal{B}$ of bounded variation function into itself.
We have
$$
\int e^{zS_n} fd\mathfrak{P}=\int \mathcal{L}_z^n(f)dx.
$$
From this, using the fact that the map $z\rightarrow \mathcal{L}_z$ is complex analytic on $\mathbb{C}$ and $\mathcal{L}_z$ is a smooth perturbation of the compact operator $\mathcal{L}$, by   Kato's perturbation theorem we will get as asymptotic expansion of the leading eigenvalue of $\mathcal{L}_z$ in the neighborhood of $z=0$ and this will be enough to show  the CLT and the large deviation.  
 We follow here the paper \cite{ANV}, where the reader could find all the explications and the references for this technique, especially since \cite{ANV} treats an annealed perturbation like ours.
 We first compute the variance. A direct computation gives
$$
\sigma^2=\lim_{n\rightarrow \infty}\frac1n\mathbb{E}_{\mathfrak{P}}(S_n^2)=\int \left(\int u^2(x+s(x)\epsilon)d\psi(\epsilon)\right)d\mu(x)+
$$
$$
2\sum_{n=1}^{\infty}\int \int u(x+s(x)\epsilon)d\psi(\epsilon)\ \mathcal{U}^n\left[\int u(x+s(x)\epsilon)d\psi(\epsilon)\right]d\mu(x).
$$
Since the operator $\mathcal{U}$ is the adjoint of the averaged operator $\mathcal{L}$ and the latter has a spectral gap  because it is compact, the correlation integral in the second sum above decays exponentially fast in $n$ for $u\in \mathcal{B}\cap L^1.$ Therefore the variance will be finite; we will further ask  that is non zero.
\begin{proposition}
Let us suppose that our Markov systems verifies {\bf assumption TM} and the modulation term is bounded as in (\ref{56}). Moreover suppose now $\sigma^2>0$ and take $u\in \mathcal{B}\cap L^1$ centered, namely 
$$
\int \int u(x+s(x)\epsilon) d\psi(\epsilon)d\mu(x)=0.
$$
Then we have:
    \begin{itemize}
    
\item (CLT) The process $\frac{S_n}{\sqrt{n}}$ converges in law to $\mathcal{N}(0, \sigma^2)$ under the probability $\mathfrak{P}.$
\item (LD) There exists a non-negative function $c$, continuous, strictly convex, vanishing only at $0,$ such that for sufficiently small $\varepsilon$ we have
$$
\lim_{n\rightarrow \infty}\frac1n \log \mathfrak{P}(S_n>n\varepsilon)=-c(\varepsilon).
$$

    \end{itemize}
    \end{proposition}
    \section{Concentration inequalities}\label{CI}
We now give another large deviation result provided by a concentration inequality for our process, which is an interesting  general result in itself. We continue to consider   $Z_n$ as the sum of a Markov chain $X_n$ plus the noise $s(X_n)\epsilon_n$ as we did in the filtering section.
 For a function $K(y_1,\dots,y_n)$ which is separately Lipschitz in each variable $y_j$  with Lipschitz constant $\text{Lip}_j(K),$ we need  to estimate
$$
\mathfrak{C}_n:=\mathbb{E}_{\mathbb{P}\times\mathbb{P}_{\epsilon}}\left(e^{K(X_1+s(X_1)\epsilon_1,\dots,X_n+s(X_n)\epsilon_n)-\mathbb{E}_{\mathbb{P}\times\mathbb{P}_{\epsilon}}(K(X_1+s(X_1)\epsilon_1,\dots,X_n+s(X_n)\epsilon_n))}\right)
$$
and find an exponential bound for it. Notice that we now admit that the chain and the observational noise are defined on different spaces with probability respectively $\mathbb{P}$ and $\mathbb{P}_{\epsilon}.$ Inspired by \cite{maldonado}, we can rewrite the preceding expectation as 
$$
\mathfrak{C}_n=\int d\mathbb{P}e^{\mathbb{E}_{\mathbb{P}_{\epsilon}}(K(X_1+s(X_1)\epsilon_1,\dots,X_n+s(X_n)\epsilon_n))}e^{-\mathbb{E}_{\mathbb{P}\times\mathbb{P}_{\epsilon}}(K(X_1+s(X_1)\epsilon_1,\dots,X_n+s(X_n)\epsilon_n))}
$$
$$
\cdot\int d\mathbb{P}_{\epsilon}e^{K(X_1+s(X_1)\epsilon_1,\dots,X_n+s(X_n)\epsilon_n)-\mathbb{E}_{\mathbb{P}_{\epsilon}}(K(X_1+s(X_1)\epsilon_1,\dots,X_n+s(X_n)\epsilon_n))}
$$
where the first integral in $d\mathbb P$ integrates all that it is following it. \\

{\bf Assumption CI.} We now suppose that $s$ is  a constant function equal to $s_M.$ \\

Since the variables $\epsilon_j$ are i.i.d. we can apply the concentration bound to the inner integral above and conclude that there is a constant $C_{\epsilon}$  such that 
$$
\int d\mathbb{P}_{\epsilon}e^{K(X_1+s_M\epsilon_1,\dots,X_n+s_M\epsilon_n)-\mathbb{E}_{\mathbb{P}_{\epsilon}}(K(X_1+s_M\epsilon_1,\dots,X_n+s_M\epsilon_n))}\le e^{C_{\epsilon} \sum_{j=1}^n s_M^2 \text{Lip}_j(K)^2}.
$$
On the other hand in the first integral in $\mathfrak{C}_n$ we recognize the concentration integral of the separately Lipschitz function
$$
(y_1, \dots, y_n)\rightarrow \mathbb{E}_{\mathbb{P}_{\epsilon}}(K(y_1+s_M\epsilon_1,\dots,y_n+s_M\epsilon_n)),
$$
with $j$- Lipschitz constant $\text{Lip}_j(K).$ 
Then we use the fact that an exponential concentration inequality holds for Markov chain with exponential convergence to the (unique) stationary measure $\mu,$ see, for instance, \cite{paulin} 
Therefore there will be another constant $C_X$ such that 
$$
\int d\mathbb{P}  e^{\mathbb{E}_{\mathbb{P}_{\epsilon}}(K(X_1+s_M\epsilon_1,\dots,X_n+s_M\epsilon_n))}e^{-\mathbb{E}_{\mathbb{P}\times\mathbb{P}_{\epsilon}}(K(X_1+s_M\epsilon_1,\dots,X_n+s_M\epsilon_n))}\le  e^{C_{X} \sum_{j=1}^n \text{Lip}_j(K)^2}.
$$
Combining the two bounds we get a very useful large deviation result by applying Chebyshev's inequality:
\begin{proposition}
Let us suppose that our Markov systems verifies {\bf assumption TM} and the modulation term satisfies Assumption {\bf CI}; then  we have:
    $$
\mathbb{P}\times \mathbb{P}_{\epsilon}\left(|K(X_1+s_M\epsilon_1,\dots,X_n+s_M\epsilon_n)-\mathbb{E}_{\mathbb{P}\times \mathbb{P}_{\epsilon}}(K(X_1+s_M\epsilon_1,\dots,X_n+s_M\epsilon_n))|>t\right)\le
$$
$$
2\exp\left(\frac{-t^2}{4\sum_{j=1}^n \mbox{Lip}_j(K)^2[C_{\epsilon}s_M^2+C_X]}\right).
$$
\end{proposition}
\subsection{Application: deviation of the empirical measure.} \label{EEMM} Take $\mathcal{G}$ the set of Lipschitz functions on $I_{\Gamma}$ with Lip constant at most $1.$ By ergodicity we know that 
$$
\lim_{n\rightarrow \infty}\frac1n\sum_{j=0}^{n-1}g(X_j+s_M\epsilon_j)=\iint d\mu d\psi g(x+s_M\epsilon)=:\int g(y)d\rho_Z(y),
$$
where $\rho_Z$ is the distribution introduced in section \ref{filll}.
Write the {\em empirical measure}:
$$
\mathcal{E}_n(\Omega, \epsilon):=\frac1n\sum_{j=0}^{n-1}\delta_{X_j(\Omega)+s_M\epsilon_j(\epsilon)},
$$
and then put, for two probability measures $\nu_1, \nu_2$ with support  on $I_{\Gamma}:$
$$
\kappa (\nu_1, \nu_2):=\sup_{g\in \mathcal{G}}\int gd\nu_1-\int gd\nu_2,
$$
the {\em Kantorovich} distance. We thus have, since now $K(z_1, \dots, z_n)=\sup_{g\in\mathcal{G}}\frac1n\sum_{j=0}^{n-1}g(z_j)-\int gd\rho_Z$ and $\text{Lip}_j(K)\le \frac1n:$
\begin{equation}\label{zet}
\mathbb{P}\times\mathbb{P}_{\epsilon}\left(\kappa (\mathcal{E}_n,\rho_Z)>t+\mathbb{E}_{\mathbb{P}\times\mathbb{P}_{\epsilon}}(\kappa (\mathcal{E}_n,\rho_Z)\right)\le \exp\left(-\frac{nt^2}{4[c_{\epsilon}s_M^2+c_X]}\right).
\end{equation}
As expected the presence of the heteroscedastic factor $s_M$ makes the rate of convergence slower.
Notice that this result is in principle valid if we compare the empirical measure with {\em any} other probability measure, but it becomes  meaningful by using the stationary distribution $\rho_Z$ just because the quantity $\mathbb{E}_{\mathbb{P}\times\mathbb{P}_{\epsilon}}(\kappa (\mathcal{E}_n,\rho_Z))$ is exponentially small. 
\section{Extreme value theory}\label{EEVVTT}
\subsection{The extreme value distribution}\label{nb}
Let us consider  a measurable  function $\phi:I_{\Gamma}\rightarrow \mathbb{R}\cup\pm \{\infty\},$ 
the process $Z_n$ defined in (\ref{rfu}) and construct the new process $\phi\circ Z_n$ as
$
\phi(T^{n}_{\overline{\eta}}x+s(T^{n}_{\overline{\eta}}x)\epsilon_{n}).
$
We are interested in the {\em extreme value distribution}:
$$
W_n=\mathfrak{P}(M_n\le u_n),
$$
where 
$$
M_n=\max_{0\le j \le n-1}\{\phi(T^{j}_{\overline{\eta}}x+s(T^{j}_{\overline{\eta}}x)\epsilon_{j})\}.
$$
Let us introduce the set, which we will call "ball" for reason which will become evident later on:
$$
B_n:=\{\phi\ge u_n\},
$$
and \\
{\bf Assumption B}: suppose that the diameter of $B_n$  goes to zero when $n\rightarrow \infty$ (this is clearly an assumption on $\phi$ as well).\\

We defer to \cite{book} for a discussion of the choice of  observable $\phi$  and their relationship with the three extreme value distributions: {\em Gumbel}, {\em Weibull} and {\em Fr\'echet}.\\

Define also the quantity
$$
\frak{n}_{B_n}:=\inf\{j\ge 1; T^{j}_{\overline{\eta}}x+s(T^{j}_{\overline{\eta}}x)\epsilon_{j}\in B_n\},
$$
which gives the {\em first hitting time} to the ball $B_n.$ By stationarity of $\mathfrak{P},$ it is easy to show that 
\begin{equation}\label{VIP}
\mathfrak{P}(\mathfrak{n}_{B_n}>n)=W_n=\mathfrak{P}(M_n\le u_n),
\end{equation}
which establishes an important link between the law of extremes and the hitting time in small sets.
We  begin to compute the quantity
\begin{equation}\label{fevt}
W_n=\mathfrak{P}(Z_0\in B_n^c, Z_1\in B_n^c,\dots Z_{n-1}\in B_n^c),
\end{equation}
as
\begin{align*}
W_n&=\iiint h(x)1_{B_n^c}(x+s(x)\epsilon_0)1_{B_n^c}(T_{\overline{\eta}}x+s(T_{\overline{\eta}}x)\epsilon_1)\dots \\
&\quad\quad\dots 1_{B_n^c}(T^{n-1}_{\overline{\eta}}x+s(T^{n-1}_{\overline{\eta}}x)\epsilon_{n-1})dx d\tilde{\mathbb{P}}(\overline{\eta})d\tilde{\mathbb{P}}(\overline{\epsilon}),\label{mevt}
\end{align*}
where $h$ is the density of $\mu.$
Since the $\epsilon_k$ are independent, we continue as:
\begin{align*}
W_n&=\iint h(x) \int 1_{B_n^c}(x+s(x)\epsilon_0)d\psi (\epsilon_0)\int 1_{B_n^c}(T_{\overline{\eta}}x+s(T_{\overline{\eta}}x)\epsilon_1)d\psi(\epsilon_1)\dots \\
&\quad\dots\int 1_{B_n^c}(T^{n-1}_{\overline{\eta}}x+s(T^{n-1}_{\overline{\eta}}x)\epsilon_{n-1})d\psi (\epsilon_{n-1})d\tilde{\mathbb{P}}(\overline{\eta})dx.
\end{align*}
In order to handle with the previous quantity we now introduce the following operator, for $f$ in our Banach space $\mathcal{B}$:
\begin{align*}
\hat{\mathcal{L}_n}f(x)&:=\mathcal{L}\left(f(\cdot) \int 1_{B_n^c}((\cdot)+s(\cdot)\epsilon)d\psi (\epsilon)\right)(x)\\
&=\int \mathcal{L}_{\eta}\left(f(\cdot) \int 1_{B_n^c}((\cdot)+s(\cdot)\epsilon)d\psi (\epsilon)\right)(x)d\Theta(\eta),
\end{align*}
where $\mathcal{L}_{\eta}$ is the transfer operator of $T_{\eta}.$
We can now check that
\begin{equation}\label{vj}
W_n=\int \hat{\mathcal{L}}_n^n h(x) dx.
\end{equation}
We now show that $\hat{\mathcal{L}_n}$ is a perturbation of $\mathcal{L}$ for $n\rightarrow \infty$ 
 when the "size" of the ball suitably defined by the quantity $\Delta_n$ below  converges to zero.   Put now
$$
\int 1_{B_n^c}(x+s(x)\epsilon)d\psi (\epsilon)=\psi\left(\frac{B^c_n-x}{s(x)}\right).
$$
The "distance" between the two operators is given by
$$
\Delta_n:=\int (\mathcal{L}-\hat{\mathcal{L}_n})hdx=\int \psi\left(\frac{B_n-x}{s(x)}\right)d\mu(x).
$$

We now apply the Keller-Liverani theory \cite{KL1,KL2}\footnote{We also refer to chapter 7 of the book \cite{book} and to the monograph \cite{last} for the details of such a theory.}. In order to apply it, we need the compactness of the unperturbed operator $\mathcal{L},$ which we already established, and also the Lasota-Yorke (LY) inequality for the perturbed operator $\hat{\mathcal{L}_n}$. This follows the same proof of the LY inequality for $\mathcal{L}$ in \cite{lillo21} by using the property of the Markov kernel $p$  established after eq. (\ref{MMKK}).   Moreover, having defined
$$
\pi_n=\sup_{f\in \mathcal{B}, \|f\|_{\mathcal{B}}\le1}\left|\int (\mathcal{L}-\hat{\mathcal{L}_n})fdx\right|,
$$
we should prove that $\pi_n$ goes to zero for $n\rightarrow \infty$ and moreover there is a constant $c_{\pi}$ such that
\begin{equation}\label{poi}
\pi_n \|(\mathcal{L}-\hat{\mathcal{L}_n})h\|_{\mathcal{B}}\le c_{\pi}|\Delta_n|.
\end{equation}
But $\left|\int (\mathcal{L}-\hat{\mathcal{L}_n})fdx\right|=\left|\int f(x)\left( \int 1_{B_n}(x+s(x)\epsilon)d\psi (\epsilon)\right)dx\right|\le \|f\|_{\mathcal{B}} \int \psi\left(\frac{B_n-x}{s(x)}\right)dx.$ We now make the following:\\
{\bf Assumption S:} We ask that:
$$
\lim_{n\rightarrow \infty}    \int \psi\left(\frac{B_n-x}{s(x)}\right)dx=0. 
$$
This assumption will guarantee that $\pi_n\rightarrow 0$ and moreover it will allow us to establish (\ref{poi}), provided the density $h$ of the stationary measure is strictly positive on the set $\{x; x+s(x)\epsilon \in B_n\}.$\footnote{The term $\|(\mathcal{L}-\hat{\mathcal{L}_n})h\|_{BV}$ in (\ref{poi}) is in fact bounded by a constant by the LY.}
We therefore get that the top eigenvalue $\iota$ of $\hat{\mathcal{L}_n}$ differs from $1,$ which is the top eigenvalue of $\mathcal{L}$ by
$$
1-\iota= \beta\Delta_n+o(\Delta_n), 
$$
where   $\beta$ is the extremal index
given by 
$$
\beta=1-\sum_{k=0}^{\infty}q_k,
$$
where
$$
q_k:=\lim_{n\rightarrow \infty}\frac{\int(\mathcal{L}-\hat{\mathcal{L}_n}) \hat{\mathcal{L}}_n^k(\mathcal{L}-\hat{\mathcal{L}_n})(h)dx}{\Delta_n},
$$
provided the limit exists.
We now introduce the threshold condition:\\

{\bf Assumption S'}: When $n\rightarrow \infty$ there exists a positive number $\tau$ such that 
\begin{equation}\label{ss}
 \ n \Delta_n\rightarrow \tau,
\end{equation}Since the top eigenvalue of $\hat{\mathcal{L}}_n^n$ behaves as $\iota^n$ and using the scaling (\ref{ss}), by replacing in (\ref{vj}) we finally   get:
\begin{proposition}(Extreme Value Distribution)
Let us suppose that our Markov system verifies {\bf Assumption TM} and the observational noise satisfies {\bf Assumptions S, S'}. Moreover suppose that the extremal index $\beta>0.$ Then we have
$$
W_n\rightarrow e^{-\beta \tau}, \ n\rightarrow \infty.
$$
\end{proposition}

\begin{remark}
   Assumptions {\bf S}  and {\bf S'} requires that the two integrals $ \int \psi\left(\frac{B_n-x}{s(x)}\right)dx,     \int \psi\left(\frac{B_n-x}{s(x)}\right)d\mu(x)$ go to zero when $n\rightarrow \infty.$ An easy way to check them is to consider $s(x)=constant.$ In this case the first integral above is simply bounded by the Lebesgue measure of the ball $B_n$ (by the translation invariance of the Lebesgue measure), while the second integral is bounded in the same way since $h\in L^{\infty},$ then we use {\bf Assumption  B}. In this setting we also require that $h$ is strictly bounded from below on the set $B_n-s\epsilon.$
\end{remark}
\subsection{The extremal index}
Using the definition of the various operators, it is straightforward to check that
$$
q_k=
$$
$$
\lim_{n\rightarrow \infty}\frac{\int d\mu(x)d\tilde{\mathbb{P}}(\underline{\eta})d\tilde{\mathbb{P}}(\underline{\epsilon})1_{B_n}(x+s(x)\epsilon_0)1_{B_n^c}(T_{\overline{\eta}}x+s(T_{\overline{\eta}}x)\epsilon_1)\dots 1_{B_n^c}(T^{k}_{\overline{\eta}}x+s(T^{k}_{\overline{\eta}}x)\epsilon_{k})1_{B_n}(T^{k+1}_{\overline{\eta}}x+s(T^{k+1}_{\overline{\eta}}x)\epsilon_{k+1}) }{\Delta_n}.
$$
We conjecture that all the $q_k=0,$ the contrary could happen if the ball  returns to  itself which is quite unlikely due to the presence of two noises. There is an easy interpretation of $q_k:$ it weights occurrences which start in $B_n,$ spend $k$ times outside $B_n$ and return in $B_n$ at the $k+1$ time.
There is however a case where we could prove that all the $q_k$ are zero:\\

{\bf Assumption EI}: we take  $s$ as a constant, equal to  $1$ to simplify the next considerations and we will also suppose  that the measure $\psi$ is not atomic. \\

Then we will  assume the point of view of the Markov chain which we already used in section \ref{CI}. In this case and using the notations of that section we see that the denominator $\frak{d}_k$ in the expression for  $q_k$ could be written as
$$
\frak{d}_k=\mathbb{P}\times\mathbb{P}\left(X_0+\epsilon_0\in B_n, X_1+\epsilon_1\in B^c_n, \dots X_k+\epsilon_k\in B^c_n, X_{k+1}+\epsilon_{k+1}\in B_n\right)
$$
We integrate now with respect to the first factor by considering
$$
\frak{d}'_k:=\mathbb{P}\left(X_0\in B_n-\epsilon_0, X_1\in B^{c}_n-\epsilon_1, \dots X_k\in B^{c}_n-\epsilon_k, X_{k+1}\in B_n-\epsilon_{k+1}\right).
$$
As in eq. (39) section 6 in (\cite{lillo22}), we have now
$$
\frak{d}'_k\le \int 1_{B_n-\epsilon_{k+1}}(x_{k+1})p(x_k, x_{k+1})dx_{k+1}\int dx_k\mathcal{L}^{k}(h1_{ B_n-\epsilon_{0}})(x_k),
$$
where $p(\cdot, \cdot)$ is the Markov kernel and we showed in \cite{lillo22} that $p$ is uniformly bounded by a constant $C_p.$ Then
$$
\frak{d}'_k\le C_p \text{Leb}(B_n-\epsilon_{k+1})\mu( B_n-\epsilon_{0}).
$$
By integrating with respect to the second factor measure $\mathbb{P}_{\epsilon}$ we  get
$$
\frak{d}_k\le C_p \int dx \psi(B_n-x)\int d\mu(y)\psi(B_n-y).
$$
By dividing with $\Delta_n$ we finally get
$$
\frac{\frak{d}_k}{\Delta_n}\le \frac{C_p\int dx \psi(B_n-x)\int d\mu(y)\psi(B_n-y)}{\int d\mu(y)\psi(B_n-y)},
$$
and when $n\rightarrow \infty$ the ratio goes to zero proving that $q_k=0.$
We have thus shown
\begin{proposition}\label{pp1}
Let us suppose that our Markov system verifies {\bf Assumption TM} and the observational noise satisfies {\bf Assumptions EI}. Moreover suppose the observable $\phi$ enjoys {\bf Assumption B}; then 
$$
W_n\rightarrow e^{- \tau}, \ n\rightarrow \infty.
$$
\end{proposition}

  In the context of observational noise perturbing a deterministic dynamical system, a result similar to Proposition \ref{pp1} was proved in \cite{FV}.


             \section{Poisson statistics}\label{PPSS}
Using the spectral approach recently proposed in \cite{AFGV} we can now get the distribution of the number of visits of our random orbits, even perturbed with observational noise, in the set $B_n$ verifying {\bf Assumption B}  and in a suitable renormalized interval of time. We will argue  that such a distribution is a standard Poisson.\\
The starting point is therefore  to compute the following distribution
\begin{equation}\label{PS}
\mathfrak{P}\left(\sum_{i=0}^{n'}1_{B_n}(T_{\overline{\eta}}^i x+s(T_{\overline{\eta}}^i x)\epsilon_i)=k\right),
\end{equation}
where  
\begin{equation}\label{rez}
n':=\left\lfloor{\frac{\tau}{\int \psi((B_n-x)/s(x))d\mu(x)}}\right\rfloor.
\end{equation}
Let us call $S_{n'}$ the quantity inside the parenthesis in (\ref{PS}): we will compute its distribution.
In particular 
we will compute its characteristic function $\Phi(S_n')$ given by
$$
\Phi_{S_{n'}}(\lambda)=\int e^{i\lambda S_{n'}}d\mathfrak{P}.
$$
Let us now introduce a new operator $\tilde{\mathcal{L}_n}$ defined as, on functions $f$ belonging to our Banach space of bounded variation functions:
$$
\tilde{\mathcal{L}_n}(f)(x)=\mathcal{L}(f \int e^{i\lambda(1_{B_n}(\cdot+s(\cdot)\epsilon)}d\psi(\epsilon))(x).
$$
Notice again that this operator is a small perturbation of $\mathcal{L}$ and therefore we could proceed with the perturbative technique as before. It is now straightforward to check that
$$
\Phi_{S_{n'}}(\lambda)=\int \tilde{\mathcal{L}}_n^{n'}(f)(x) d\mu(x).
$$
The top eigenvalue $\upsilon$ of $\tilde{\mathcal{L}_n}$ differs from $1,$ which is the top eigenvalue of $\mathcal{L}$ by
$$
1-\upsilon= \l(\lambda)\Delta'_n+o(\Delta_n'),
$$
where $\Delta'_n$ denotes as before the distance between the two operators and is given by
$$
\Delta_n':=\int (\mathcal{L}-\tilde{\mathcal{L}_n})hdx=(1-e^{i\lambda})\int \psi((B_n-x)s(x)^{-1})d\mu(x)=(1-e^{i\lambda})\Delta_n;
$$
moreover 
$$
l(\lambda)=1-\sum_{k=0}^{\infty}q'_k(\lambda),
$$
and 
$$
q'_k(\lambda):=\lim_{n\rightarrow \infty}\frac{\int(\mathcal{L}-\tilde{\mathcal{L}_n}) \tilde{\mathcal{L}}_n^k(\mathcal{L}-\tilde{\mathcal{L}_n})(h)dx}{\Delta'_n},
$$
provided the limit exists. In such a case, 
since $\Phi_{S_{n'}}(\lambda)$ grows like $\upsilon^{n'},$ by the scaling (\ref{rez})  on $n'$ we finally have:
$$
\Phi_{S_{n'}}(\lambda)\rightarrow e^{-l(\lambda)(1-e^{i\lambda})\tau}, \ n\rightarrow \infty.
$$
Supposing that $l(\lambda)$ is continuous at $0,$ we see that the right hand side of the previous equation is the characteristic function of a random variable $W$ which we could identify with the limit, in distribution, of $W_k.$ It turns out that $W$ is discrete and infinitely divisible, so that it is {\em compound Poisson distributed.} By the same argument used in the preceding section, we could prove that when $s=constant$   then   all the $q'_k(\lambda)=0$ and therefore 
 $l(\lambda)=1.$  We have therefore proved:
\begin{proposition}\label{pp2}
Let us suppose that our Markov system verifies {\bf Assumption TM} and the observational noise satisfies {\bf Assumptions EI}. Moreover suppose the observable $\phi$ enjoys {\bf Assumption B}. Then we  have that $l(\lambda)=1,$ and  $W$ has the standard Poisson distribution $$\mathcal{P}(\tau):=\frac{\tau^ke^{-\tau}}{k!}.$$ 
\end{proposition}

\appendix
\section{An application from finance}

As we said above, our perturbed map (\ref{rrr}) is linked to the concept of systemic risk in finance. If we rewrite it as a nonlinear autoregressive model as in \cite{lillo22}  
\begin{equation}\label{eq:differenceequation}
    \phi_n = T(\phi_{n-1}) + \sigma_{\n}(\phi_{n-1})Y_{n-1}, 
\end{equation}
we have that  $\phi_n$, $n \in \mathbb{N}_{\geq 1}$ is a sequence of real numbers in a bounded interval of $\mathbb{R}$, $T$ is a deterministic map on $[0,1]$ perturbed with the additive and heteroscedastic noise $\sigma_{\n}(\phi_{n-1})Y_{n-1}$, being $\n \in \mathbb{N}_{\geq 1}$ a parameter that modulates the intensity of the noise; $\n$ is such that one retrieves the deterministic dynamic as $\n \rightarrow \infty$. Finally, $Y_n$, $n \in \mathbb{N}_{\geq 1}$, is a sequence of independent and identically distributed (\textit{i.i.d.}) real-valued random variables defined on some filtered probability space. 
In this setting, $\phi_n$ in \eqref{eq:differenceequation} represents the suitably scaled financial leverage of a representative investor (a bank) that invests in a risky asset. At each point in time, the scaling is a linear function of the leverage itself. The bank's risk management consists of two components. First, the bank uses past market data to estimate the future volatility (the risk) of its investment in the risky asset. Second, the bank uses the estimated volatility to set its desired leverage. However, the bank is allowed maximum leverage, which is a function of its perceived risk because of the Value-at-Risk (VaR) capital requirement policy. More specifically, the representative bank updates its expectation of risk at time intervals of unitary length, say $(n, n + 1]$ with $n\in\mathbb{N}_{\geq 1}$, and, accordingly, it makes new decisions about the leverage. Moreover, the model assumes that over the unitary time interval $(n, n + 1]$ the representative bank re-balances its portfolio to target the leverage without changing the risk expectations. The re-balancing takes place in $\n$
sub-intervals within $(n, n+ 1]$.  In particular, the
considered model is a discrete-time slow-fast dynamical system.\\
It is interesting to quote precisely the  form of the map $T,$ which has an apparently complicated analytic structure, but is ultimately a unimodal map:
$$
\phi_n= T(\phi_{n-1})=\frac{A(\phi_{n-1})-1}{\gamma_0+cA(\phi_{n-1})},
$$
where 
$$
 A(u)\overset{\text{def}}{=}\frac{1+\gamma_0 u}{[\omega(1-c u)^2 + \overline{\Sigma}_{\epsilon}(1-u)^{-2}(1+\gamma_0 u)^2]^{1/2}},
$$
being $\gamma_0, \omega, c, \overline{\Sigma}_{\epsilon}$ parameters related to the financial model, see \cite{lillo22} for details and figure 1 for the graph of $T$ taken from our previous work \cite{lillo22}\footnote{The value of $\alpha$ given in Fig. 1 depends on the return distribution and VaR constraint entering the definition of $\phi_n,$ see \cite{lillo22} for more details.}.\\

 \begin{figure}
    \centering
    \includegraphics[scale=0.55]{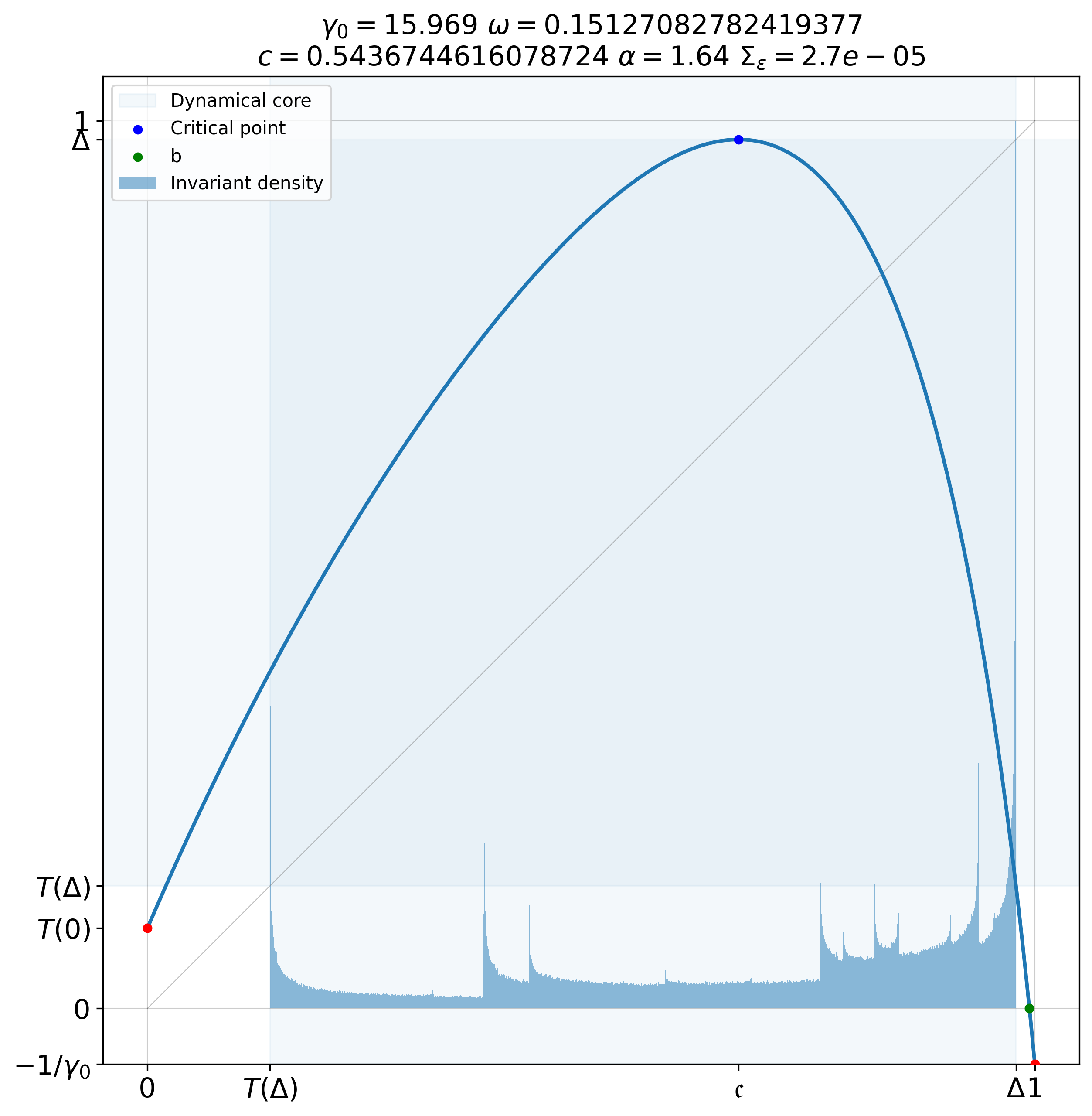}
    \caption[Deterministic Component]{Taken from \cite{lillo22}: plot of the deterministic component $T(\phi)$, $\gamma_0=15.969$, $\alpha=1.64$, $\Sigma_{\epsilon}=2.7\times 10^{-5}$. The value for $\gamma_0$ is taken from the empirical analysis in \cite{lillo21}, Section 7.2, (where it is denoted simply by $\gamma$) . The value $\alpha=1.64$ corresponds to a VaR constraint of $5\%$ in case of a Gaussian distribution for the returns. The values $\Sigma_{\epsilon}=2.7\times10^{-5}$ is taken from \cite{mazzarisi2019panic}, Table 1, and corresponds to the exogenous idiosyncratic volatility at the time scale of portfolio decisions. The value for $\Omega$ and $c$ are randomly sampled from the dynamical core, once fixed the other parameters. The \textit{Blue dot} indicates the critical point $\mathfrak{c}$, the \textit{Green dot} the intersection between the map and the horizontal axis, the left-hand \textit{Red dot} indicates the image of 0, the right-hand \textit{Red dot} indicates $\lim_{\phi\rightarrow1^{-}}T(\phi)=-\frac{1}{\gamma_0}$. The support of the invariant density belongs to the so-called \textit{dynamical core} $[T(\Delta),\Delta]$.} 
    \label{fig::deterministicmap}
\end{figure}

             \section*{Acknowledgement}
             The research of SV was supported by the project {\em Dynamics and Information Research Institute} within the agreement between UniCredit Bank and Scuola Normale Superiore di Pisa and by the Laboratoire International Associ\'e LIA LYSM, of the French CNRS and  INdAM (Italy).  SV was also supported by the project MATHAmSud TOMCAT 22-Math-10, N. 49958WH, du french  CNRS and MEAE. Finally SV thanks the Great Bay University in Dongguan (China), where this work was completed.
              F.L. and S.M. acknowledge partial support by the European Program scheme ‘INFRAIA-01-2018-2019: Research and Innovation action’, grant agreement \#871042 ’SoBigData++: European Integrated Infrastructure for Social Mining and Big Data Analytics’. MT ackwnoledges the hospitality of the Centro di Ricerca Matematica Ennio De Giorgi and Universit\'a di Pisa. 

\bibliographystyle{plain}
\bibliography{references}

\begin{thebibliography}{10}

\bibitem{ANV}
R.~Aimino, M.~Nicol, and S.~Vaienti.
\newblock Annealed and quenched limit theorems for random expanding dynamical
  systems.
\newblock {\em Probability Theory and Related Fields}, 162(1):233--274, June
  2015.

\bibitem{AFGV}
J.~Atnip, G.~Froyland, C.~Gonzalez-Tokman, and S.~Vaienti.
\newblock Compound poisson statistics for dynamical systems via spectral
  perturbation.
\newblock {\em Nonlinearity}.
\newblock To appear.

\bibitem{last}
J.~Atnip, G.~Froyland, C.~Gonzalez-Tokman, and S.~Vaienti.
\newblock Thermodynamic formalism and perturbation formulae for quenched random
  open dynamical systems.
\newblock {\em Dissertationes Mathematicae}.
\newblock To appear.

\bibitem{BDM}
J.~Brockner and G.~Del~Magno.
\newblock Asymptotic stability of the optimal filter for random chaotic maps.
\newblock {\em Nonlinearity}, 30:1809, 2017.

\bibitem{CMR}
O.~Capp{\'e}, E.~Moulines, and T.~Ryd{\'e}n.
\newblock {\em Inference in Hidden Markov Models}.
\newblock Springer Series in Statistics. Springer, New York, 2005.

\bibitem{CH}
P.~Chigansky, R.~Liptser, and R.~Van~Handel.
\newblock Intrinsic methods in filter stability.
\newblock In {\em The Oxford Handbook of Nonlinear Filtering}, pages 319--351.
  Oxford University Press, Oxford.

\bibitem{FV}
D.~Faranda and S.~Vaienti.
\newblock Extreme value laws for dynamical systems under observational noise.
\newblock {\em Physica D}, 280--281:86--94, 2014.

\bibitem{HH}
H.~Hennion and L.~Herv{\'e}.
\newblock {\em Limit Theorems for Markov Chains and Stochastic Properties of
  Dynamical Systems by Quasicompactness}, volume 1766 of {\em Lecture Notes in
  Mathematics}.
\newblock Springer, 2001.

\bibitem{sei}
H.~Kantz and T.~Schreiber.
\newblock {\em Nonlinear Time Series Analysis}.
\newblock Cambridge University Press, 2 edition, 2004.

\bibitem{KL1}
G.~Keller.
\newblock Rare events, exponential hitting times and extremal indices via
  spectral perturbation.
\newblock {\em Dynamical Systems}, 27(1):11--27, 2012.

\bibitem{KL2}
G.~Keller and C.~Liverani.
\newblock Rare events, escape rates and quasistationarity: some exact formulae.
\newblock {\em Journal of Statistical Physics}, 135:519--534, 2009.

\bibitem{otto}
S.~P. Lalley.
\newblock Beneath the noise, chaos.
\newblock {\em Annals of Statistics}, 27(2):461--479, 1999.

\bibitem{sette}
S.~P. Lalley.
\newblock Removing the noise from chaos plus noise.
\newblock In {\em Nonlinear Dynamics and Statistics}, pages 233--244.
  Birkh{\"a}user, Boston, MA, 2001.

\bibitem{LN}
S.~P. Lalley and A.~B. Nobel.
\newblock Denoising deterministic time series.
\newblock {\em Dynamics of Partial Differential Equations}, 3(4):259--279,
  2006.

\bibitem{lillo21}
F.~Lillo, G.~Livieri, S.~Marmi, A.~Solomko, and S.~Vaienti.
\newblock Analysis of bank leverage via dynamical systems and deep neural
  networks.
\newblock {\em SIAM Journal on Financial Mathematics}, 14(2):598--643, 2023.

\bibitem{lillo22}
F.~Lillo, G.~Livieri, S.~Marmi, A.~Solomko, and S.~Vaienti.
\newblock Unimodal maps perturbed by heteroscedastic noise: an application to a
  financial systems.
\newblock {\em Journal of Statistical Physics}, 190:156, 2023.

\bibitem{book}
V.~Lucarini, D.~Faranda, A.~M. Freitas, J.~M. Freitas, M.~Holland, T.~Kuna,
  M.~Nicol, M.~Todd, and S.~Vaienti.
\newblock {\em Extremes and Recurrence in Dynamical Systems}.
\newblock Wiley, New York, 2016.

\bibitem{maldonado}
C.~Maldonado.
\newblock Fluctuation bounds for chaos plus noise in dynamical systems.
\newblock {\em Journal of Statistical Physics}, 2012.

\bibitem{mazzarisi2019panic}
P.~Mazzarisi, F.~Lillo, and S.~Marmi.
\newblock When panic makes you blind: A chaotic route to systemic risk.
\newblock {\em Journal of Economic Dynamics and Control}, 100:176--199, 2019.

\bibitem{paulin}
D.~Paulin.
\newblock Concentration inequalities for markov chains by marton couplings and
  spectral methods.
\newblock {\em Electronic Journal of Probability}, 2015.

\bibitem{nove}
T.~Schreiber.
\newblock Extremely simple nonlinear noise-reduction method.
\newblock {\em Physical Review E}, 47:2401--2404, 1993.

\bibitem{dieci}
T.~Schreiber and P.~Grassberger.
\newblock A simple noise-reduction method for real data.
\newblock {\em Physics Letters A}, pages 411--418, 1991.

\bibitem{thunberg2001periodicity}
H.~Thunberg.
\newblock Periodicity versus chaos in one-dimensional dynamics.
\newblock {\em SIAM Review}, 43(1):3--30, 2001.

\end{thebibliography}




\end{document}